\documentclass[12pt]{article}

\usepackage[utf8]{inputenc}
\usepackage{amsmath,amssymb,amsthm,dsfont,mathtools}

\usepackage{enumerate}
\usepackage{hyperref}
\usepackage{cleveref}

\theoremstyle{plain}
\newtheorem{theorem}{Theorem}
\numberwithin{theorem}{section}
\newtheorem{proposition}[theorem]{Proposition}
\newtheorem{corollary}[theorem]{Corollary}
\newtheorem{lemma}[theorem]{Lemma}
\newtheorem*{claim}{Claim}
\theoremstyle{definition}
\newtheorem*{definition*}{Definition}

\DeclarePairedDelimiter{\Bra}{\langle}{\rangle}
\DeclarePairedDelimiter{\norm}{\lVert}{\rVert}
\DeclarePairedDelimiter{\abs}{\lvert}{\rvert}
\DeclarePairedDelimiter{\set}{\lbrace}{\rbrace}

\DeclareMathOperator{\spn}{span}
\DeclareMathOperator{\divergence}{div}
\DeclareMathOperator{\grad}{grad}
\DeclareMathOperator{\mass}{\omega}

\DeclareMathOperator{\cicm}{C^\infty_c(\mathring M)}
\newcommand{\dLap}{\overline{\Delta}}
\newcommand{\dd}{\mathrm{d}}
\newcommand{\R}{\mathds{R}}

\title{A Lagrangian perspective on nonautonomous advection-diffusion processes in the low-diffusivity limit}
\author{Daniel Karrasch\thanks{karrasch@ma.tum.de} \and Nathanael Schilling\thanks{schillna@ma.tum.de}\and\\Zentrum Mathematik, Technische Universität München\\Boltzmannstr. 3, 85748 Garching bei München}

\begin{document}

\maketitle

\begin{abstract}
We study mass preserving transport of passive tracers in the low-diffusivity limit using
Lagrangian coordinates. Over finite-time intervals, the solution-operator of the nonautonomous
diffusion equation is approximated by that of a time-averaged diffusion equation. We show that
leading order asymptotics that hold for functions based on [Krol, 1991] extend to the dominant
nontrivial singular value. The generator of the time-averaged diffusion/heat semigroup is a
Laplace operator associated to a weighted manifold structure on the material manifold. We show
how geometrical properties of this weighted manifold directly lead to physical transport
quantities of the \emph{nonautonomous} equation in the low-diffusivity limit.
\end{abstract}


\newpage
\section{Introduction}

We begin by sketching the outlines of this paper, a more comprehensive introduction with more detailed references is given in \cref{sec:math_setting}.
We are concerned with the problem of transport and mixing in nonautonomous ad\-vec\-tion-dif\-fu\-sion processes in the vanishing-diffusivity limit.
Such processes are, in the simplest case, described by the \emph{advection-diffusion equation},
\begin{align}\label{eq:adprototype}
\partial_t u_{\varepsilon}(x,t) + \divergence(u_\varepsilon(x,t)\, V(x,t)) = \varepsilon \Delta u_\varepsilon(x,t)\,,
\end{align}
where $V$ is a time-dependent, smooth velocity field,
$u_\varepsilon$ the density of a weakly diffusive passive scalar, and $\varepsilon > 0$ is referred to as the (strength of) \emph{diffusivity}.
We sometimes omit the explicit $\varepsilon$ in our notation when referring to $u_\varepsilon$ for the sake of clarity.
In this work, we are interested in the \emph{finite-time} setting,
i.e., without loss\ $t\in\mathcal I = [0,1]$. 

Lagrangian coordinates can be obtained from the advection-only version of \cref{eq:adprototype}
with $\varepsilon = 0$.
With $p$ denoting an arbitrary point in these coordinates,
it is known that \cref{eq:adprototype} takes the form of a time-dependent diffusion (or heat) equation
\begin{align}\label{eq:diffusion}
\partial_t u_{\varepsilon}(p,t) = \varepsilon \Delta_t u_{\varepsilon}(p,t)\,.
\end{align}
The smoothly varying family of operators $\left(\Delta_t\right)_{t\in\mathcal{I}}$ may be viewed as Laplace operators on a suitably defined
time-dependent family of weighted manifolds. We want to compare the solution $u_\varepsilon$ of \cref{eq:diffusion} to the solution $\overline{u}_\varepsilon$
of the \emph{time-averaged} equation
\begin{align}\label{eq:averageddiffusion}
\partial_t \overline{u}_\varepsilon(p,t) &= \varepsilon \dLap \overline u_\varepsilon(p,t)\,, & \dLap&\coloneqq\int_0^1 \Delta_t\,\dd t\,,
\end{align}
as $\varepsilon\to 0$ at the final time $t=1$. To the best of our knowledge, an averaging approach
like this has been first taken in \cite{Press1981}, albeit in an infinite-time setting.
The operator $\dLap$ has also been introduced by Froyland in his recent work on
\emph{dynamic} isoperimetry \cite{Froyland2015a}. Consistently with his work, we will refer to $\dLap$
as the \emph{dynamic Laplacian}.

In the present work, we prove two results in the spirit of averaging theory, whose precise formulation we defer to \cref{sec:averaging}.
First, for fixed initial condition $u_0$ the final density $u_\varepsilon(1,\cdot)$ (of \cref{eq:diffusion}) is uniformly approximated by $\overline{u}_\varepsilon(1,\cdot)$
(of \cref{eq:averageddiffusion}) in leading order as $\varepsilon \rightarrow 0$; see \cref{thm:averaging1}. This result follows directly from prior work by
Krol \cite{Krol1991} on the averaging method applied to time-periodic
advection-diffusion equations, in which, by the way, the transformation to \emph{standard averaging form} is what is known as the
transformation from Eulerian to Lagrangian coordinates in continuum mechanics.
Second, we show that the largest (nontrivial) singular value/vector of the time-$1$ solution operator converge in a suitable sense to the
largest (nontrivial) eigenvalue/eigenfunction of the averaged heat semigroup defined by \cref{eq:averageddiffusion}; see \cref{thm:spectralconvergence}.

In \cref{sec:ldm}, we work towards a geometric interpretation of our averaging results within the framework of the
\emph{geometry of mixing}, as introduced in \cite{Karrasch2020c}. This leads to a strengthened 
version of Froyland's dynamic Cheeger inequality \cite{Froyland2015a}; we also draw a connection
to the notion of \emph{material barriers to diffusive transport} as developed in \cite{Haller2018,Haller2020}.
A by-product of our averaging result is an alternative and simplified proof of a low-diffusivity approximation
result for the diffusive transport across boundaries of full-dimensional material subsets; see eq.~(6) of \cite{Haller2018} and \cref{cor:Hallerflux}.
Diffusive flux or material leakage has long been implicit in different approaches in finding so-called \emph{Lagrangian coherent structures (LCSs)};
see, for instance, \cite{Haller2013a,Froyland2015a,Hadjighasem2016} and \cite{Haller2015} for a general review.
It has been identified as the potentially unifying perspective on LCSs as \emph{diffusion barriers} in \cite{Karrasch2020c},
and finally became the central object in the variational approach to \emph{material barriers to diffusive transport} in \cite{Haller2018,Haller2020}.

Our main motivation stems from transport and mixing problems as studied in physical oceanography and the atmospheric sciences.
There, a typical problem is that presumably purely advective transport processes are observed only up to some finite scale.
The effect of unresolved (small) scales is then often modelled via a weak diffusion with spatiotemporal inhomogeneity; see, for instance, \cite{Sebille2018}.
To address such problems, we treat advection-diffusion processes on (compact) smooth manifolds,
and include general time-dependent, spatially inhomogeneous and anisotropic diffusion.

We would like to emphasize that we are interested in the details of spatial inhomogeneity of mixing, that
would allow to explain significant differences in the mixing ability of different flow regions (transport/diffusion barriers vs.~enhancers). This is in contrast to asymptotic or statistical information,
like decay rates to equilibrium or spatially homogeneous effective diffusion tensors, typically obtained in
homogenization theory; see, for instance, \cite{Fannjiang1994,Pavliotis2008}.

The advection-diffusion equation \eqref{eq:diffusion} has been extensively studied in the literature.
The time-periodic case has been investigated in the low-diffusivity limit by Krol \cite{Krol1991},
cf.~also \cite{Schafer2009,Vukadinovic2015}. Time-periodic advection-diffusion problems have
also been studied by Liu \& Haller \cite{Liu2004} from the Eulerian perspective.
They developed mathematical theory for observed time-periodic patterns,
\emph{strange eigenmodes}, in periodically driven advection-diffusion processes.
The time-periodic setting is closely related to this work, as one may construct a time-periodic
advection-diffusion process from the finite-time setting by appending its ajoint equation (which is again of advection-diffusion type)
and time-periodic extension.
This resulting equation is then periodic with continuous coefficients.
Such a time-periodic extension procedure has been employed recently by Froyland \emph{et al.}~\cite{Froyland2019a} to the Fokker--Planck equation associated to a stochastic differential equation, in order to find approximations to Eulerian, spatiotemporal sets with small exponential escape
rates of stochastic trajectories.

In the \emph{autonomous} case---where $V$ in \cref{eq:diffusion} does not depend on time---semi-group theory may be applied,
and many results have been obtained in this case.  For instance, Kifer \cite[Chapter III]{Kifer1988} studies
asymptotics of spectra in the low-diffusivity limit; see also \cite[Chapter 6.7]{Freidlin2012}.
Further autonomous, non-finite-time results were obtained in \cite{Berestycki2005,Constantin2008}.

\section{Diffusion-induced Lagrangian geometries}\label{sec:math_setting}

This section is meant to be both a motivation and a gentle recall of the geometric
interpretation of advection-diffusion processes, as developed in \cite{Karrasch2020c,Haller2018}.
For a recall of fundamental differential geometry concepts and notation used, we refer to \cref{sec:diffgeo}.

\subsection{Advection processes}\label{sec:advection}
We recall some basic properties of advection processes that preserve mass; see also \cite{Haller2020}.
These generalize the notion of volume-preservation to vector fields whose flows do not preserve volume; this is done by constructing a
time-dependent volume-form $\varrho$, the (fluid) \emph{mass form}, that has precisely the property that it is preserved by the flow of a time-dependent velocity field $V$. Readers who are only interested in volume-preserving flows (with respect to a volume $\mass$, such as the usual Euclidean volume) may set $\varrho(t,x) \equiv \mass(x)$ everywhere below.

Recall that, in a fixed spatial frame, the evolution of the passively advected mass-form $\varrho$, with initial value $\mass$, is given by the advection equation/conservation law
\begin{align}\label{eq:advection}
\partial_t \varrho &= -\mathcal L_V \varrho\\
\varrho(0,\cdot) &= \mass\,.
\end{align}
Here, $V$ is the sufficiently regular time-dependent fluid (bulk) velocity.
We consider \cref{eq:advection} on an orientable smooth manifold $M$, potentially with a sufficiently regular boundary, over a (finite subset of the) finite time interval $\mathcal{I}\subset\R$. For notational simplicity, we assume w.l.o.g.~that $\mathcal{I}=[0,1]$.

\Cref{eq:advection} is well-known as a hyperbolic partial differential equation (PDE) that can be solved by the \emph{method of lines/characteristics}.
That is, consider the associated ordinary differential equation (ODE)
\begin{align}\label{eq:ODE}
\dot{x} &= v(t,x), & x(0) &= x_0\in M,
\end{align}
on $\mathcal{I}$. Let $\phi$ denote the \emph{flow map} associated to \cref{eq:ODE}, i.e., $t\mapsto\phi_0^t(x_0)$ is the unique solution of \cref{eq:ODE} satisfying the initial condition $\phi_0^0(x_0) = x_0$.
The solutions of \cref{eq:ODE} are then known as the \emph{characteristics} of \cref{eq:advection}, and each characteristic is also referred to as the \emph{trajectory} of a \emph{(fluid) particle}.
Now, with the formula for time-derivatives along trajectories, \cite[Chapter V, Prop.~5.2]{Lang1995}, \cref{eq:advection} becomes
\[
\left(\tau\mapsto(\phi_0^\tau)^* \varrho \right)'(t)  =  (\phi_0^t)^*(\mathcal L_V \varrho) + (\phi_0^t)^* \partial_t \rho  = 0\,,
\]
where $(\phi_0^t)^*$ is the pullback by $\phi_0^t$. For its push-forward $(\phi_0^t)_*$ this implies
\begin{equation}
\rho(t,\cdot) = (\phi_0^t)_* \mass\,.
\end{equation}
In particular, the mass-form along a trajectory is uniquely determined by its value anywhere on the trajectory.

In addition to the mass form $\varrho$, we would like to model the evolution of a passive tracer that is advected by the fluid.
This passive tracer is described by a (time-dependent) function/density $u$ such that the volume form $u\varrho$, integrated against any (measurable) $S \subset M$,
returns the total amount of the tracer in $S$. Here,
\begin{align}\label{eq:traceradvection}
\partial_t u &= -\dd u(V)  = -\mathcal L_V u\,, & u(0,\cdot) &= u_0\,.
\end{align}
As above, along characteristics we obtain
\begin{align*}
\frac{d}{dt}\left((\phi_0^t)^* u\right) &= (\phi_0^t)^*(\mathcal L_V u ) + (\phi_0^t)^* \partial_t u = 0 \,,
\end{align*}
therefore $u$ is constant along characteristics.

One important consequence is the following intimate relation between the PDE formulation of transport, \cref{eq:advection,eq:traceradvection},
and its ODE formulation, \cref{eq:ODE}. For any (measurable) set $S\subset M$ and any (measurable) initial scalar density $u_0$ one has
\begin{equation}\label{eq:integralconservation}
\int_Su_0\, \mass = \int_{\phi_0^t(S)}u(t,\cdot)\varrho(t,\cdot)\,.
\end{equation}
Note that \cref{eq:integralconservation} contains both the densities $\varrho$ and $u$ \emph{and} the flow map $\phi$,
which otherwise do not occur simultaneously in \cref{eq:advection,eq:traceradvection,eq:ODE}.

Next, assume the scalar is confined to some set $S\subset M$, e.g., $u_0= \mathds{1}_S$. Then, as a direct consequence of \cref{eq:integralconservation}, we have
\begin{equation}\label{eq:integralcomplement_advection}
\int_{\phi_0^t(S)^{\mathrm{c}}}u(t,\cdot)\varrho(t,\cdot) = 0\,,
\end{equation}
where $A^{\mathrm{c}}$ denotes the complement of $A$ (in $M$).
In other words, none of $u$ leaks out of the spatiotemporal tube $\bigcup_{t \in [0,1]} \phi_0^t(S)$.
For later reference and in accordance with continuum mechanics, we call any flow-invariant
spatiotemporal set $S = \bigcup_{t \in [0,1]}\phi_0^t(S_0)$ a \emph{material set}.
So far, all considerations were relative to some spatial or, synonymously, some \emph{Eulerian} frame.
Besides different spatial frames, however, which can be related to different
observers of the physical transport process, there exists the \emph{Lagrangian}
frame that is related only to the characteristics/particles of the underlying process.
Changing from \emph{some} Eulerian to \emph{the} Lagrangian frame is essentially
applying the method of lines, where one additionally declares the initial conditions of
\cref{eq:ODE}, i.e., the particles, as \emph{coordinates}, and represents all physical equations w.r.t.~those.

Briefly, in Lagrangian coordinates that are co-moving with the trajectories, \cref{eq:advection} becomes
\begin{align*}
\partial_t\varrho &= 0\,, & \partial_t u &= 0\,;
\end{align*}
\cref{eq:ODE} reads as
\begin{equation}
\dot{x} = 0\,,
\end{equation}
and, as a consequence, the ``flow map'' is the identity map for all times. \Cref{eq:integralcomplement_advection} then states that
no scalar mass leaks out of any material set into the respective complementary material set;
likewise the Lagrangian advective transport through any material surface vanishes.

\subsection{Advection-diffusion processes}\label{sec:ade}

In the following, we will consider advection-diffusion processes and re-inspect our above considerations in this framework in order to summarize the construction appearing in \cite{Karrasch2020c}.

Recall that, in a fixed spatial frame, the evolution of a weakly diffusive scalar, given by its density $u$, passively advected by a (possibly compressible)
fluid described with mass form $\varrho$ is given by the \emph{advection-diffusion equation} \cite{Landau1987,Thiffeault2003}
\begin{subequations}\label{eq:ADE}
\begin{align}
\partial_t u &= -\mathcal L_V u  + \varepsilon\divergence_\varrho (D\, \dd u)\,,\label{eq:diffscalar}\\
\partial_t\varrho &= -\mathcal L_V \varrho \,.\label{eq:fluidmass}
\end{align}
\end{subequations}
Here, $\varepsilon > 0$ is the \emph{diffusivity} (or the inverse Péclet number in non-di\-men\-sional\-ized units), which is assumed to be small, and $D\colon T^*M\to TM$ is a (possibly time-dependent) bundle morphism satisfying the following property: for given $(t,x)\in\mathcal{I}\times M$ identify $D$ with a bilinear form $\tilde{g}_t^{-1}$ on $T^*M$, then this bilinear form is symmetric and positive-definite. In particular, $D$ gives rise to a dual metric, inducing a Riemannian metric $\tilde{g}_t$ on $M$. Viewed in this sense, $D$ is a \emph{diffusion tensor field}, modeling possibly (spatially and temporally) inhomogeneous, anisotropic diffusion; for further details on the aforementioned steps, see also \cite{Karrasch2020c}.
It is also necessary to impose suitable boundary conditions in the case that the manifold $M$ has nonempty boundary $\partial M$.
We will focus on homogeneous boundary conditions of Dirichlet--, and for only some of our results, Neumann form.

Taking a closer look at \cref{eq:diffscalar}, we directly recognize
$D\,\dd u$ as the gradient of $u$ w.r.t.~the metric $\tilde{g}$. As a consequence, the diffusion term can then be elegantly represented via the Laplace operators on the family of weighted manifolds $(M,\tilde{g}_t,\theta)$,
\begin{subequations}\label{eq:ADELaplace}
\begin{align}
\partial_tu &= -\mathcal L_V u  + \varepsilon\Delta_{\tilde{\theta}_t,\tilde{g}_t}u\,,\label{eq:diffscalar2}\\
\partial_t\varrho &= -\mathcal L_V \varrho \,,\label{eq:fluidmass2}
\end{align}
\end{subequations}
where $\tilde{\theta}_t$ is the density of $\varrho$ w.r.t.~$\dd\tilde{g}_t$, i.e., $\tilde{\theta}_t\,\dd\tilde g_t =\varrho(t,\cdot)$.

In stark contrast to the advection equations, \cref{eq:advection,eq:fluidmass,eq:fluidmass2}, the advection-diffusion equation,
\cref{eq:diffscalar2}, is not amenable to the method of characteristics, and, therefore does \emph{not} introduce a concept of
deterministic particles, trajectories, or Lagrangian coordinates for the scalar $u$.
On the other hand, it is a singular perturbation of a hyperbolic PDE: namely \cref{eq:diffscalar2} with $\varepsilon=0$ as considered before.
Hence, we may introduce Lagrangian coordinates based on the characteristics of its singular limit (cf.~\cite{Thiffeault2003}), or, equivalently, based on \cref{eq:fluidmass2}.
In these Lagrangian coordinates, the advective terms in \cref{eq:ADELaplace} vanish as in \cref{sec:advection}, and we obtain from the well-known
pullback transformation rules
\begin{equation}
\partial_t u = \varepsilon\Delta_{\theta_t,g_t} u = \varepsilon\divergence_{\mass}(g_t^{-1}\,\dd u)\,,\label{eq:preLagrangianADE}
\end{equation}
which is an evolution equation on the material manifold $M$. Here, $g_t\coloneqq (\phi_0^t)^*\tilde{g}_t$ is the diffusion-adapted
pullback metric on $M$ and $\theta_t = (\phi_0^t)^* \theta$. As a consequence of mass preservation, the volume form
$\mass$---w.r.t.~which we compute the divergence---does not depend on time. Henceforth, we write $\divergence$ without
a subscript whenever we refer to $\divergence_{\mass}$. Moreover, let $\Delta_t \coloneqq \divergence (g_t^{-1} \dd u)$, then
with this notation \cref{eq:preLagrangianADE} simplifies to
\begin{equation}
\partial_t u = \varepsilon\Delta_t u\,.\label{eq:LagrangianADE}
\end{equation}

The lack of characteristics for the advection-diffusion equation has another, crucial consequence:
given a (proper) material subset $S\subset M$, the amount of $u$ is in general no longer constant over time, or, equivalently
\begin{equation}\label{eq:integralcomplement}
T_0^t(S,u_0) \coloneqq \int_{S} u_0\,\mass -  \int_{\phi_0^t(S)}u(t,\cdot)\,\varrho(t,\cdot)  \neq 0\,.
\end{equation}
In simple words, there is leakage of $u$ out of or into material sets. For given scalar fields and material subsets, the associated
scalar leakage is an non-trivial and interesting quantity when regarded as a function of material sets $S$, see \cite{Haller2018,Haller2020}.

In Lagrangian coordinates, \cref{eq:integralcomplement} reads as
\begin{equation}\label{eq:integralcomplement2}
T_0^t(S,u_0) = \int_{S}\left(u_0 - u(t,\cdot)\right)\,\mass \,.
\end{equation}
Furthermore, assuming for the moment that all involved functions are sufficiently smooth, differentiating with respect to $t$ and applying the fundamental theorem of calculus yields
\begin{equation}\label{eq:time_t_flux}
T_0^t(S,u_0) = -\varepsilon\int_0^t\left( \int_{S} \Delta_\tau u(\tau,\cdot)\, \mass\right) \dd\tau\, .
\end{equation}
Heuristically, for very small $\varepsilon$ we have $u \approx u_0$ which suggests that
\begin{equation}\label{eq:heuristicallyspeaking}
\frac{T_0^t(S,u_0)}{\varepsilon} \approx -\int_0^t \int_{S}\Delta_\tau u_0 \, \mass\, \dd\tau\, \eqqcolon \overline T_0^t(S,u_0)\, .
\end{equation}
Indeed, it was shown in \cite{Haller2018,Haller2020} for the case that $M \subset \R^n$ and homogeneous Neumann boundary condition  that
\begin{equation}\label{eq:leadingordertransport}
T_0^t(S,u_0) = \varepsilon \overline T_0^t(S,u_0) + o(\varepsilon),\quad \varepsilon \to 0\,. 
\end{equation}
In \cref{sec:averaging}, we develop an alternative proof of \cref{eq:leadingordertransport} on compact manifolds with Dirichlet boundary; see \cref{cor:eq6}.

By Fubini's theorem, we have that
\[
\overline T_0^t(S,u_0) = -\int_{S}\int_0^t \Delta_\tau u_0\, \dd\tau\,\mass = -\int_{S}\int_0^t \Delta_\tau \,\dd\tau\, u_0\mass\,.
\]
For $t=1$, this suggests the definition
\begin{equation}\label{eq:dynamic_laplace}
\overline \Delta \coloneqq  \int_0^1 \Delta_t\,\dd t\,.
\end{equation}
This operator was recently introduced in \cite{Froyland2015a,Froyland2017} and coined \emph{dynamic Laplacian}.
With this notation, \cref{eq:leadingordertransport} reads as
\begin{equation*}
T_0^1(S,u_0) = -\varepsilon \int_{S} \overline \Delta u_0 \, \mass + o(\varepsilon)\,,
\end{equation*}
and combines mathematical tools from recent work on material surfaces that extremize diffusive flux \cite{Haller2018,Haller2020} on the one hand,
and dynamic isoperimetry \cite{Froyland2015a,Froyland2017} on the other hand.
A goal of this work is to investigate these connections rigorously; see \cref{cor:Hallerflux}.

\subsection{The geometry of mixing and diffusive permeability}\label{sec:geometryofmixing}

Our study is centered around the \emph{geometry of mixing} as induced by
$\dLap$ and introduced in \cite{Karrasch2020c}. There, it was observed
that $\dLap$ is the Laplace operator of a specific \emph{weighted}
(Riemannian) manifold. With the above notation, let us define
\begin{equation}
\overline g = \left(\int_{\mathcal I}g_t^{-1} \dd t \right)^{-1}\,.
\end{equation}

\begin{lemma}[{\cite[Prop.~3]{Karrasch2020c}}]
The dynamic Laplacian $\overline \Delta$ is the Laplace operator associated to
the weighted manifold $(M,\overline g, \theta)$, where $\theta\,\dd\overline g = \mass$.
\end{lemma}

As in \cite{Karrasch2020c}, we refer to the material manifold $M$, equipped with the metric $\overline g$ and density $\theta$ as the \emph{geometry of mixing}.
This, together with \cref{sec:diffgeo}, shows that the geometry of mixing is constructed to have the following elegant properties:
(i) volume/fluid mass is given by $\mass$, the differential form preserved by the flow, and
(ii) diffusion is given by averaged pullback diffusion tensors as featured in the dynamic Laplacian.
It was further observed in \cite{Karrasch2020c} that
\begin{align}\label{eq:averagedLADE}
\partial_t \overline u &= \varepsilon \overline \Delta \overline u, & \overline u(0,\cdot) &= u_0\,,
\end{align}
is an \emph{averaged} (cf.~\cite{Sanders2007,Pavliotis2008}) form of \cref{eq:LagrangianADE}.
It was conjectured that \cref{eq:averagedLADE} approximates \cref{eq:LagrangianADE} in the vanishing diffusivity limit, leaving open the concrete nature of the approximation and the required assumptions.
We prove this in \cref{sec:averaging} building on a similar result in the classic, i.e., time-periodic, averaging context \cite{Krol1991}.
A by-product of our averaging result is a new proof of \cref{eq:leadingordertransport}, as mentioned above.
We also prove that the convergence extends to singular values/vectors, addressing an open question from \cite{Karrasch2020c}.

To summarize the previous sections: \cref{eq:leadingordertransport} shows that in leading order as $\varepsilon \rightarrow 0$, the diffusive transport out of a material set is determined by
\[
\overline T_0^1(S,u_0) = -\int_S \overline \Delta u_0 \,\mass\,.
\]
By the divergence theorem, we have
\begin{equation}\label{eq:smoothareasurfacetransport}
\overline T_0^1(S,u_0) = -\int_{\partial S} \dd u_0(\overline \nu)\, \dd \overline A \,,
\end{equation}
where $\overline \nu$ is the outward-pointing $\overline g$-unit normal vector field on $\partial S$ and
$\dd \overline A = \theta \dd A_{\overline g}$ is the induced area form on $\partial S$ in the geometry of mixing.

Of course, $\overline T_0^1(S,u_0)$ could be represented similarly in other weighted geometries on $M$:
choose any metric $\widetilde{g}$, compute the density $\widetilde{\theta}$ of the fluid mass relative to the induced volume $\dd g$, denote the induced area form and the $\widetilde g$-unit normal vector field on $\partial S$ by $\dd\widetilde{A}$ and $\widetilde\nu$, respectively, then one obtains analogously to \cref{eq:smoothareasurfacetransport}
\begin{equation}\label{eq:genericflux}
\overline T_0^1(S,u_0) = -\int_{\partial S} \dd u_0(H\widetilde \nu)\, \dd \widetilde A \,,
\end{equation}
where $H = \overline{g}^{-1} \widetilde{g}$ is a tangent space isomorphism; cf.~also \cref{sec:sagm} below and
references \cite{Haller2018,Haller2020}, in which $\overline T_0^1(S,u_0)$ is represented in the usual
Euclidean/physical geometry, and $H$ is coined the \emph{transport tensor} (denoted by
$\mathbf{T}_{t_0}^{t_1}$ there). It is exactly the absence of any additional tensor in
\cref{eq:smoothareasurfacetransport} that makes arguably the geometry of mixing the ``best-adapted''
or ``most natural'' geometry in which to look at leading-order diffusive flux in Lagrangian coordinates.

\Cref{eq:smoothareasurfacetransport} emphasizes that, in leading order, the
diffusive transport $T_0^1(S,u_0)$ out of a material set $S$ depends
on
(i) the differential/gradient of the initial concentration $u_0$ along $\partial S$, and
(ii) properties of the geometry of mixing via the surface measure $\dd\overline A$ and unit normal vector field $\overline \nu$.
As argued in \cite{Karrasch2020c}, $\dd \overline A$ is particularly interesting as an intrinsic
measure of the ``diffusive permeability'' of the material boundary $\partial S$. In many physical applications, it is of great interest to
diagnose the mixing structure of an advection-dominated transport process \emph{independent} of any specific scalar quantity; cf.~the discussion in \cite{Haller2018}.

\section{Finite-time averaging of the advection-dif\-fu\-sion equation}\label{sec:averaging}

We now show that in the setting of the advection-diffusion equation, the diffusion process induced by the dynamic Laplacian
approximates the diffusion of the advection-diffusion equation in Lagrangian coordinates, in the limit of vanishing diffusivity.

In this section, we restrict to those $M$ that are compact manifolds whose boundary, if it is nonempty, is smooth.
The proof can be extended to other classes of manifolds also, provided there is a suitable maximum principle.

\subsection{Uniform convergence}\label{sec:averaging_pointwise}

Let $u_\varepsilon\colon M \times [0,1] \rightarrow \R$ solve the advection-diffusion equation in Lagrangian coordinates for diffusivity $\varepsilon$
with initial condition $u_0\colon M \rightarrow \R$ with---if there is a boundary---homogeneous Dirichlet or Neumann\footnote{Given a metric for each time $t \in [0,1]$, we also have a $g_t$ unit-normal vector $\nu_t$ field on $\partial M$ for each $t \in [0,1]$. The natural homogeneous Neumann condition is thus $\mathrm{d}u_\varepsilon(t,\cdot)(\nu_t) = 0$ for each $t \in [0,1]$ on $\partial M$.} boundary conditions.
Thus, in the interior of $M$, $u_\varepsilon$ satisfies
\begin{align}\label{mlade}
\partial_t u_\varepsilon &= \varepsilon \Delta_t u_\varepsilon\,.
\end{align}
Similarly, let $\overline u_\varepsilon\colon M \times [0,1] \rightarrow \R$ be the solution of the heat flow generated by the dynamic Laplacian $\dLap$,
with initial condition $u_0$ and diffusivity $\varepsilon$, i.e.,
\begin{align}\label{mlavade}
\partial_t \overline u_\varepsilon &= \varepsilon \dLap \overline u_\varepsilon
\end{align}
and---if there is a boundary---homogeneous Dirichlet or Neumann\footnote{Here, we require $\mathrm d\overline u_\varepsilon(t,\cdot)(\overline \nu) = 0$ independent of $t$, where $\overline \nu$ is unit normal field for $\overline g$ on $\partial M$.} boundary conditions of the corresponding  type.
We will focus mainly on the case of Dirichlet boundary, and refer to  \cref{sec:parabolicpdes} for a
recall of results regarding existence, uniqueness and regularity of solutions. We expect analogous existence, uniqueness and regularity
results to hold in the Neumann case on manifolds, but could not find a reference.

\begin{definition*}
Depending on the boundary condition used, we call an initial value $u_0 \in C^\infty(M)$ \emph{admissible} if
(i) $u_0$ is compactly supported in the interior of $M$ (Dirichlet case),
(ii) if $u_0$ is constant in a neighborhood of the spatial boundary $\partial M$ (Neumann case).
\end{definition*}

This definition is motivated by the fact that the time-dependent parabolic \cref{mlade,mlavade} may not be smooth at $t=0$ if the initial value $u_0$
does not satisfy certain compatibility conditions at the boundary; see \cite[Sect.~7.1, Thm.~6]{Evans2010}; cf.~also \cite{Krol1991}.
These may differ between \cref{mlade} and \cref{mlavade}. Our definition of
admissibility guarantees that the compatibility conditions of both the time-dependent and the averaged equations are satisfied simultaneously.

\begin{proposition}\label{thm:averaging1}
With $u_\varepsilon$ and $\overline{u}_\varepsilon$ as above, let $u_0$ be an admissible initial value.
Then
\begin{equation}\label{eq:averaging1}
u_\varepsilon(1, x) = \overline u_\varepsilon(1, x) + O(\varepsilon^2)\,, \qquad \varepsilon\to 0,
\end{equation}
uniformly in $x$.
\end{proposition}

\begin{proof}
The proof is a simplification of the one given in \cite{Krol1991}. Let
\begin{align*}
\widetilde u_\varepsilon = u_0 + \varepsilon \int_0^t \Delta_\tau u_0\, \dd\tau\,.
\end{align*}
We start with the Dirichlet boundary condition case.
Let $\mathcal L^{\varepsilon} \coloneqq \partial_t - \varepsilon \Delta_t$ and
observe that $\mathcal L^\varepsilon \widetilde u_\varepsilon = -\varepsilon^2 \int_0^t \Delta_t \Delta_s u_0\, \dd s$.
As $u_0$ is smooth and $M$ compact,
\[
C \coloneqq \sup_{t \in [0,1]}\left\|{\left(\int_0^t \Delta_t \Delta_\tau u_0\, \dd\tau\right)}\right\|_{L^\infty(M)} < \infty\,.
\]
By definition, $\mathcal L^\varepsilon u^\varepsilon = 0$.
Thus, $\mathcal L^\varepsilon(\widetilde{u}_\varepsilon- u_\varepsilon - C\varepsilon^2 t) \leq 0$; $u_\varepsilon$ and $\widetilde u_\varepsilon$ agree at $t=0$;
and (by the admissibility of the initial value) both satisfy Dirichlet boundary conditions. The weak maximum principle (\cref{lemma:wmp}) therefore yields that
\[
\max_{[0,1]\times M}(\widetilde u_\varepsilon - u_\varepsilon - C\varepsilon^2t) = \max_{[0,1]\times \partial M\cup \lbrace{0\rbrace\times M}}( - C\varepsilon^2t)\leq 0\,.
\]
As a consequence, we have $\max_{[0,1]\times M}(\widetilde u_\varepsilon - u_\varepsilon) \leq  C\varepsilon^2$.

One may prove $(u_\varepsilon - \widetilde u_\varepsilon)\leq  C\varepsilon^2$ along the same lines. Thus, $\norm{u_\varepsilon - \widetilde u_\varepsilon}_{L^\infty([0,1] \times M)} = O(\varepsilon^2)$. For $t=1$, this implies the uniform expansion
\begin{equation}\label{eq:taylorexp}
u_\varepsilon(1,\cdot) = \widetilde{u}_{\varepsilon}(1,\cdot) + O(\varepsilon^2) = u_0 + \varepsilon \overline \Delta u_0 + O(\varepsilon^2)\,.
\end{equation}
The right-hand side coincides up to second order with the expansion of $\overline{u}_\varepsilon(1,\cdot) (= \exp(\varepsilon\dLap)u_0)$
which yields the claim.

For Neumann boundary conditions, the proof goes along the same lines, where the
weak maximum principle must be augmented with the parabolic Hopf boundary
point lemma (see \cite[Chapter 3, Thm.~6]{Protter1984}) to ensure that a strict
maximum cannot be achieved at positive time.
\end{proof}

We restate \cref{eq:taylorexp} for further reference, and also observe that it can be
interpreted as the time-continuous generalization of \cite[Thm.~5.1]{Froyland2015a}.

\begin{corollary}
\label{cor:taylorexp}
Under the assumptions of \cref{thm:averaging1},
\begin{align}
u_\varepsilon(1,x) &= u_0(x) + \varepsilon\dLap u_0(x) + O(\varepsilon^2)\,,
\end{align}
uniformly in $x$.
\end{corollary}

\begin{corollary}
\label{cor:taylorexp2}
Under the assumptions of \cref{thm:averaging1},
\[
u_\varepsilon(1, \cdot) = u_0 + \varepsilon\dLap u_0 + O(\varepsilon^2)\,,
\]
in $L^p(M,\mass)$ for all $p \in [1,\infty]$.
\end{corollary}
\begin{proof}
For $p=\infty$, our claim corresponds to \cref{cor:taylorexp}. For $p\in[1,\infty)$, the natural injection $L^\infty(M,\mass) \hookrightarrow L^p(M,\mass)$
is well-defined and continuous since $\mass(M)$ is finite, which yields the claim.
\end{proof}

\begin{corollary}\label{cor:eq6}
Under the assumptions of \cref{thm:averaging1},
\begin{equation}\label{eq:transportapproximation}
T_0^1(S, u_0)=\int_S u_0(x)\, \mass - \int_S u_\varepsilon(x,1)\, \mass = -\varepsilon \int_S \dLap u_0\, \mass + O(\varepsilon^2)\,.
\end{equation}
\end{corollary}

\begin{proof}
This follows by integrating \cref{eq:taylorexp} over $S$.
\end{proof}

\subsection{Convergence of singular values}\label{sec:conv_sing_vals}

We denote by $P^\varepsilon_t$ and $\overline P^\varepsilon_t$ the time-$t$ solution operators of \cref{mlade,mlavade}, respectively, i.e.,
$u_\varepsilon(t,\cdot) = P^\varepsilon_t u_0$ and $\overline{u}_\varepsilon(t,\cdot)=\overline{P}^{\varepsilon}u_0$.
To reduce notational clutter, we write $P^\varepsilon \coloneqq P^\varepsilon_1$ and $\overline P^\varepsilon \coloneqq \overline P^\varepsilon_1$.
We only treat homogeneous Dirichlet boundary in this section.

The previous section dealt with the relationship of $P^\varepsilon$ and $\overline P^\varepsilon$ in the limit $\varepsilon \rightarrow 0$.
In particular, \cref{thm:averaging1}, with this notation, is
\begin{align}\label{eq:averaging1revisited}
\norm{(P^\varepsilon - \overline P^\varepsilon)u_0}_{L^\infty(M)} &= O(\varepsilon^2) \quad \textnormal{ for\ all } u_0 \in \cicm\,.
\end{align}
Recall that as $P^\varepsilon$ is compact (cf.~\cref{sec:parabolicpdes}), the first singular value of $P^\varepsilon$
is given by the operator norm of $P^\varepsilon\colon L^2(M,\mass) \rightarrow L^2(M,\mass)$. By \cref{lemma:l2contr}, $P^\varepsilon$ is a contraction on $L^2(M,\mass)$,
hence $\norm{P^\varepsilon}\leq 1$.

If $M$ is boundaryless, then $P^\varepsilon \mathds{1}_M = \mathds{1}_M$, and,
as a consequence, $\norm{P^\varepsilon}=1$ for any $\varepsilon>0$.
Since the subspace of constant functions is a trivial invariant subspace, we restrict the domain of $P^\varepsilon$ to its orthogonal complement, the space of mean-free functions. If $M$ has a boundary, we consider $P^\varepsilon$ with its domain the entire $L^2(M,\omega)$.

With these preparations, we denote the largest \emph{nontrivial} singular value by $\sigma^\varepsilon$, and a corresponding (normalized)
left singular vector by $v^\varepsilon$, i.e.,
\begin{align*}
\norm{v^\varepsilon}_{L^2(M,\mass)} &= 1, & \norm{P^\varepsilon} = \norm{P^\varepsilon v^\varepsilon}_{L^2(M,\mass)} &= \sigma^\varepsilon \,.
\end{align*}
For the sake of brevity, let $\norm{\cdot}_0 \coloneqq \norm{\cdot}_{L^2(M,\mass)}$ and $\Bra{\cdot,\cdot}_0 = \Bra{\cdot,\cdot}_{L^2(M,\mass)}$.
\Cref{eq:averaging1revisited} suggests the conjecture that
\begin{align}\label{eq:singularvaluenorms}
\abs{\norm{P^\varepsilon} - \norm{\overline P^\varepsilon}} = o(\varepsilon)\,,
\end{align}
where the norm is the operator norm. By the spectral mapping theorem (see, for instance, \cite[Sect.~1, Thm.~2.4(c)]{Pazy1983}),
\begin{align*}
\norm{\overline P^\varepsilon} &= \mathrm{e}^{\varepsilon \overline \lambda} = 1 + \varepsilon \overline \lambda + o(\varepsilon)\,,
\end{align*}
where $\overline \lambda<0$ is the largest, i.e., smallest in absolute value,
nontrivial eigenvalue of the dynamic Laplacian. Thus, \cref{eq:singularvaluenorms}
is equivalent to
\begin{align}
\norm{P^\varepsilon} = 1 + \varepsilon \overline \lambda + o(\varepsilon)\,.
\end{align}
It can be interpreted as an expansion of the first singular value of $P^\varepsilon$ in $\varepsilon$,
in analogy to the expansion obtained in \cref{cor:taylorexp}. We will prove the following equivalent statement.

\begin{theorem}\label{thm:spectralconvergence}
With the above notation and assuming a Dirichlet boundary, one has
\begin{align}\label{eq:toprove_spectrum}
\lim\limits_{\varepsilon \rightarrow 0}\frac{ \sigma^\varepsilon - 1}{\varepsilon}  &= \overline \lambda\,.
\end{align}
\end{theorem}

\begin{proof}
We split the proof into several steps.

\textbf{Step 1}: We start by proving the lower bound
\begin{align}\label{eq:lowerbound}
\liminf\limits_{\varepsilon \rightarrow 0} \frac{\sigma^\varepsilon - 1}{\varepsilon} \geq \overline \lambda\,.
\end{align}
To this end, the operator-norm definition of $\sigma^\varepsilon$ shows that $\sigma^\varepsilon \geq \norm{P^\varepsilon u}_0$ for all $u \in \cicm$
in the domain of $P^\varepsilon$ which have $\norm{u}_{0} = 1$.
Applying \cref{cor:taylorexp2} to such $u$ yields $P^\varepsilon u = u + \varepsilon \overline \Delta u + o(\varepsilon)$ in $L^2(M,\mass)$.
Therefore, we also have $\norm{P^\varepsilon u}_{0}^2 = \norm{u}_0^2 + 2\varepsilon\Bra{u,\overline \Delta u} + o(\varepsilon)$. Since $\norm{u}_0 = 1$,
we obtain $\liminf\limits_{\varepsilon \rightarrow 0}\frac{(\sigma^\varepsilon)^2 - 1}{\varepsilon} \geq 2\Bra{u,\overline \Delta u}$. The right hand side can be made arbitrarily close to $2 \overline \lambda$, which shows
\begin{align}\label{eq:squareeq}
\liminf\limits_{\varepsilon \rightarrow 0}\frac{(\sigma^\varepsilon)^2 - 1}{\varepsilon} \geq 2\overline \lambda\,.
\end{align}
From \cref{lemma:l2contr} it follows that $0 \leq \sigma^\varepsilon \leq 1$. Thus
$\sigma^\varepsilon \rightarrow 1$ for $\varepsilon\to 0$.  Finally, as $(\sigma^\varepsilon)^2 - 1 = (\sigma^\varepsilon - 1)(\sigma^\varepsilon + 1)$,
we deduce \cref{eq:lowerbound} from \cref{eq:squareeq}.

\textbf{Step 2}: We now prove the upper bound,
\begin{align}\label{eq:upperbound}
\limsup\limits_{\varepsilon \rightarrow 0} \frac{\sigma^\varepsilon - 1}{\varepsilon} \leq \overline \lambda\,,
\end{align}
which is somewhat more involved.
It is based on the identity:
\begin{align}\label{eq:importanteq}
\frac{(\sigma^\varepsilon-1)(\sigma^\varepsilon+1)}{\varepsilon} = \frac{\norm{P^\varepsilon v^\varepsilon}_0^2 - \norm{v^\varepsilon}_0^2}{\varepsilon}  &= 2\int_0^1 \Bra{v^\varepsilon(t),\Delta_t v^\varepsilon(t)}_0\,\dd t\,,
\end{align}
where the first equality is satisfied as $v^\varepsilon$ is first non-trivial singular vector, $v^\varepsilon(t) \coloneqq P^\varepsilon_t(v^\varepsilon)$, and the second equality is a direct consequence of the fundamental theorem of calculus applied to $f^\varepsilon(t) \coloneqq \Bra{v^\varepsilon(t), v^\varepsilon(t)}_0$.
To connect \cref{eq:importanteq} to the theory of elliptic partial differential equations, in analogy to \cite{Evans2010} we introduce the bilinear form
\[
a_t(u,w) \coloneqq -\Bra{u,\Delta_t w}\,,
\]
defined (by unique continuous extension) for $u,w \in H^1_0(M,g,\mass)$, where $g$ is an arbitrary fixed metric (e.g. $g_0$) used to measure lengths and angles.
The So\-bo\-lev space $H_0^1(M,g,\mass) \subset L^2(M,\mass)$ is defined as the Hilbert space with norm $\norm{\cdot}_1^2 \coloneqq \norm{\cdot}_0^2 + |\cdot|^2_1$,
here $|\cdot|_1$ is induced by the bilinear form
\begin{align}
\Bra{u,v}_1 &\coloneqq \int_M g(\grad_g u, \grad_g v)\,\mass
\end{align}
using the metric $g$ and volume-form $\mass$. This norm is equivalent to the usual $H^1(M,g,\dd g)$ Sobolev norm since $\mass$ is smooth and nonvanishing on the compact manifold $M$. As usual, $H^1_0(M,g,\mass)$ is defined as the completion of $\cicm$ w.r.t.~the norm $\norm{\cdot}_1$.
We have shown in Step 1 that $\sigma^\varepsilon \rightarrow 1$ for $\varepsilon\to0$. As a consequence, \cref{eq:importanteq} is equivalent to
\begin{align}\label{eq:liminf2}
\beta \coloneqq \liminf\limits_{\varepsilon \rightarrow 0} \frac{1 - \sigma^\varepsilon}{\varepsilon} &= \liminf\limits_{\varepsilon \rightarrow 0} \int_0^1 a_t(v^\varepsilon(t), v^\varepsilon(t))\,\dd t\,.
\end{align}
\Cref{eq:liminf2} is the negative of the left hand side of \cref{eq:upperbound}.
The bilinear form $a_t(\cdot,\cdot)$ on $H^1_0(M,g,\mass)$ is positive, continuous and coercive (cf.~\cref{lemma:uniform_parabolicity}), and thus induces
a norm $\norm{\cdot}_{a_t}$that is equivalent to $|\cdot|_1$. In particular, $\norm{\cdot}_{a_t}$-continuous functionals are $|\cdot|_1$-continuous functionals and vice versa. Therefore, the weak topologies for these norms coincide. The Banach-Steinhaus theorem,
with the norm $\norm{\cdot}_{a_t}$, thus states that if $u_n \rightarrow u$  weakly in $H^1_0(M,g,\mass)$, then
\begin{align}\label{eq:banachsteinhaus}
a_t(u,u) \leq \liminf\limits_{n \rightarrow 0} a_t(u_n,u_n)\,.
\end{align}

We are now in a position to prove \cref{eq:toprove_spectrum} by contradiction. To do so, we will employ a construction similar to the ``direct method'' from the calculus of variations; cf.~\cite{Gelfand1963}.
To this end, we take a null sequence $(\varepsilon_n)_{n \in \mathds{N}}$ for which $\int_0^1 a_t(v^{\varepsilon_n}(t), v^{\varepsilon_n}(t))\,\dd t$ converges to $\beta$. Assume, for the sake of contradiction, that
\begin{align}\label{eq:forcontradiction}
\beta = \lim\limits_{n \rightarrow \infty } \int_0^1 a_t(v^{\varepsilon_n}(t),v^{\varepsilon_n}(t))\, < -\overline \lambda\,.
\end{align}
We will use a claim whose proof we defer:

\begin{claim}
	There exist $v \in H^1_0(M,g,\mass)$ with $\norm{v}_0 = 1$ and a subsequence of $(\varepsilon_n)_n$, for simplicity again denoted by $(\varepsilon_n)_n$, for which the sequences $(v^{\varepsilon_n}(t))_n$ converge weakly in $H^1_0(M,g,\mass)$ to $v$ for every $t\in[0,1]$.
\end{claim}

For this specific $v$, Fatou's lemma and \cref{eq:banachsteinhaus} imply that
\begin{align}\label{eq:almost_contradiction}
\int_0^1 a_t(v,v)\, \dd t \leq \liminf\limits_{n \rightarrow \infty}\int_0^1 a_t(v^{\varepsilon_n}(t), v^{\varepsilon_n}(t))\,\dd t = \beta  < -\overline \lambda\,.
\end{align}
The left hand side, in a weak sense, is equal to $-\Bra{v,\overline \Delta v}_0$, the bilinear form associated to the weak form of the dynamic Laplacian.
It is well-known that the Rayleigh quotient $v\mapsto-\Bra{v ,\overline \Delta v }_0/\Bra{v,v}_0$ is minimized by $-\overline \lambda$ on $H^1_0(M,g,\mass)$; see, for instance, \cite[Sect.~6.5, Thm.~2]{Evans2010}.
With $\norm{v}_0=1$, \cref{eq:almost_contradiction} states that $v$'s Rayleigh quotient is strictly lower, hence a contradiction.

It follows that $\beta \geq -\overline \lambda$, we conclude using \cref{eq:liminf2} that
\begin{align*}
\liminf\limits_{\varepsilon \rightarrow 0} \frac{1 - \sigma^\varepsilon}{\varepsilon} \geq -\overline \lambda\,,
\end{align*}
which proves Step 2.

\textbf{Step 3, proof of claim}:
Our proof requires that there exists $\varepsilon_0 > 0$ such that $C \coloneqq \sup_{0\leq \epsilon<\varepsilon_0, t \in [0,1]} \abs{v^\varepsilon(t)}_1$ is finite, this part is done in \cref{sec:spectrum_proofs}.

Assuming that $C$ is finite, the Rellich-Kondrachev theorem \cite[Sect.~4, Prop.~3.4]{Taylor2011},
states that $v^{\varepsilon_n}(0) \rightarrow v$ in $L^2(M,\mass)$ (up to passing to a subsequence if necessary), and therefore $\norm{v}_0 = 1$.
After again passing to a subsequence if necessary, we may assume $v^{\varepsilon_n}(0) \rightarrow v\in H^1_0(M,g,\mass)$ weakly in $H^1(M,g,\mass)$ by the (sequential) Banach-Alaoglu theorem; see, for instance, \cite[Thm.~3.2.1]{Buhler2018}.

To show that this limit is attained by $v^\varepsilon(t)$ also for $t \neq 0 $ as $\varepsilon \rightarrow 0$, we differentiate $h^\varepsilon(t) \coloneqq \norm{v^\varepsilon(t) - v^\varepsilon(0)}_0^2$, and apply the fundamental theorem of calculus to yield
\[
\norm{v^\varepsilon(t)-v^\varepsilon(0)}_0^2 = 2\varepsilon\left\lvert\int_0^t a_\tau( v^\varepsilon(\tau), v^\varepsilon(\tau) - v^\varepsilon)\,\dd\tau\right\rvert \leq 4\varepsilon C^2C'\,,
\]
where
\[
C' \coloneqq \sup_{t\in [0,1],~u,w \in H^1_0(M,g,\mass)}|a_t(u,w)|/(|u|_1|w|_1) < \infty\,;
\]
see \cref{lemma:uniform_parabolicity}.
We may apply the fundamental theorem of calculus due to the absolute continouity ensured by \cite[Sect.~5.9, Thm.~3]{Evans2010}, see also \cref{sec:parabolicpdes}.
As $v^{\varepsilon_n}(0) \rightarrow v$, it follows that $v^{\varepsilon_n}(t) \rightarrow v$ in $L^2(M,\mass)$ for all $t \in [0,1]$. In particular, $v$ is the only $L^2(M,g,\mass)$ accumulation point in
the set $F \coloneqq \left\{v^{\varepsilon_n}(t)\right\}_{n\in\mathds{N}, t \in [0,1]}$, therefore also the only weak $H^1(M,g,\mass)$ accumulation point.
The sequential Banach-Alaoglu theorem guarantees that the set $F$ is weakly sequentially compact in $H^1(M,g,\mass)$. Combining this with the fact that $v$ is its only accumulation point yields weak convergence of $v^\varepsilon(t) \rightarrow v$ in $H^1(M,g,\mass)$ for all $t \in [0,1]$.

This finishes the proof of \cref{thm:spectralconvergence}.
\end{proof}

\subsubsection{Convergence of eigenvectors}

The proof of \cref{thm:spectralconvergence} also shows that the corresponding
eigenvectors must converge in $L^2$ (in fact, even weakly in $H^1$).
Since, in general, the singular vectors of $P^\varepsilon$ satisfy different
compatibility conditions at the boundary to those of $\overline P^\varepsilon$,
this is somewhat surprising.

\section{Diffusive transport and surface area}\label{sec:ldm}

In this section, we look at properties of the surface area form $\dd \overline A$ in the geometry
of mixing, and how it relates to other, similar, area forms obtained from different types of averaging.

In the setting of the advection-diffusion equation, we have assumed that the time set
$\mathcal I$ is the unit interval equipped with the Lebesgue measure.
For the purpose of this section (only), we may weaken this assumption towards
$(\mathcal{I},\dd t)$ being a probability space, such as a finite set of numbers
equipped with the normalized counting measure, or a compact interval equipped
with the Lebesgue measure normalized by the interval's length.
By the term \emph{surface}, we refer to a smooth, oriented, embedded (codimension-1) submanifold.

\subsection{Surface area in the geometry of mixing}\label{sec:sagm}

Let $g$ be any metric on the material manifold $M$, we call $g$ the ``reference metric''. This could be, for instance, some
``universal'' spatial metric (the way we measure lengths and volume), defined on $M$, or any of the diffusion-adapted metrics from $(g_t)_{t\in\mathcal{I}}$.
The choice of $g$ is in analogy to the choice of local coordinates in differential
geometry -- we will derive expressions for various quantities in terms of $g$. The
metric $g$ is in no way required to be related to the physical transport process under consideration.
In particular, if $g$ is the Euclidean metric in some coordinate chart,
we obtain coordinate representations in that chart.

As before, define a mass-induced surface area form $\dd A$ on any surface $\Gamma\subset M$ via
$\iota_\nu\mass$, where $\nu$ is the $g$-unit normal vector field\footnote{
If $\Gamma$ is the boundary of a full-dimensional submanifold, we take the outward-pointing unit normal
}

With this notation, $\overline C \coloneqq \overline g^{-1} g$ and $C_t \coloneqq g^{-1}_t g$ are tangent bundle isomorphisms, i.e., $\overline{C},C_t\colon TM \rightarrow TM$. Then
\[
\overline C = \left(\int_{\mathcal I}g_t^{-1}\,\dd t\right)g = \int_{\mathcal I} C_t\,\dd t\,.
\]
For $v \in T_xM \subset TM$, we have that
\begin{multline*}
\norm*{\overline Cv}_{\overline g}^2 = \overline g\left(\overline Cv,\overline Cv\right) = \left[\overline g\left(\overline Cv\right)\right]\left(\overline Cv\right) = \\
\left[\overline g\,\overline g^{-1} gv\right]\left(\overline Cv\right) = g(v)\left(\overline Cv\right)= g\left(v,\overline Cv\right).
\end{multline*}
Denote by $\nu_t$, $t\in[0,1]$, and $\overline{\nu}$ the unit normal vector fields w.r.t.~$g_t$ and $\overline{g}$ on $\Gamma$.
As with the reference metric, we define
\begin{align}\label{eq:time_t_area}
\dd A_t &\coloneqq \iota_{\nu_t}\mass\,, & \dd \overline{A} &\coloneqq \iota_{\overline\nu}\mass\,.
\end{align}
In other words, corresponding to the three types of metrics---reference $g$, time-dependent $(g_t)_t$ and time-averaged $\overline{g}$---we derive three area forms ($\dd A$, $\dd A_t$, and $\dd\overline{A}$) from the mass form.

We now show how to relate to each other area form that are induced by the mass
form via different metrics on a surface $\Gamma$.

\begin{lemma}\label{lemma:surfacechange}
Let $g,\tilde g$ be metrics on $M$. Let $\Gamma$ be a surface in $M$, $\tilde C \coloneqq \tilde g^{-1} g$, and $\nu$ and $\tilde \nu$ their respective (consistently oriented) unit normal vector fields on $\Gamma$. Then
\[
\iota_{\tilde\nu}\mass = g(\nu, \tilde \nu)\,\iota_{\nu}\mass = g\left(\nu,\tilde C\nu\right)^{1/2}\,\iota_{\nu}\mass\,.
\]
\end{lemma}

\begin{proof}
The first equality is trivial, because we may represent $\tilde\nu$ as the linear combination of $g(\nu, \tilde \nu)\nu$ and its projection onto $T_p\Gamma$. But the latter does not contribute to the result.
It remains to show $g(\nu,\tilde\nu)=g(\nu,\tilde C\nu)^{1/2}$. To this end, we show that $\tilde\nu = g(\nu,\tilde C\nu)^{-1/2}\,\tilde C\nu$. First, observe that $\tilde C \nu$ is $\tilde g$-normal to $T_p\Gamma$, since for any $v\in T_p\Gamma$ we have
\[
\tilde g(\tilde C \nu, v) = (\tilde g \tilde g^{-1} g \nu)(v) = g(\nu,v) = 0\,.
\]
Now, $\norm{\tilde C \nu}_{\tilde g}^2 = \tilde g\left(\tilde C \nu, \tilde C\nu\right) = g\left(\nu,\tilde C\nu\right)$,
which means that $g(\nu,\tilde C\nu)^{-1/2}\,\tilde C\nu$ is also $\tilde g$-normalized. Finally,  $g(\nu,\tilde C\nu)^{-1/2}\,\tilde C\nu = \tilde\nu$
necessarily as they share the same orientation: $g(\nu, \tilde \nu) = \tilde g (\tilde C \nu, \tilde C \nu) > 0$.
\end{proof}

Applying \cref{lemma:surfacechange} to the metrics $g$ and $g_t$, we obtain
\begin{equation}\label{timetarea}
\dd A_t  = \sqrt{g(\nu,C_t\nu)}\,\dd A\,,
\end{equation}
and for $g$ and $\overline g$,
\begin{equation}\label{samix}
\dd\overline{A} = \sqrt{g\left(\nu,\overline C\nu\right)}\,\dd A\,.
\end{equation}

By combining \cref{lemma:surfacechange} with \cref{cor:eq6}, we obtain the approximation result for accumulated diffusive flux through boundaries of full-dimensional material submanifolds.

\begin{corollary}[{\cite[eq.~(6)]{Haller2018}\footnote{Recall that $C_t = g_t^{-1}g$, where $g$ is here the Euclidean metric on the flat state space, and corresponds to the transport tensor in \cite{Haller2018,Haller2020}; and $\nu$ is the outward-pointing $g$-unit normal vector field on $\partial S$. In \cite{Haller2018}, material surfaces are considered that are not necessarily the boundary of a full-dimensional set. In case they are, \cite[Eq.~(6)]{Haller2018} measures the \emph{influx}, which explains the opposite sign to ours. They also require weaker technical assumptions and obtain a slightly weaker result than that $O(\varepsilon^2)$ error appearing here. }}]\label{cor:Hallerflux}
Let $S\subset M$ be a full-dimensional submanifold with smooth boundary, and $u_0$ an admissible initial condition. Then
\[
T_0^1(S, u_0) = -\varepsilon\int_0^1\int_{\partial S} \dd u_0(C_t \nu)\,\dd A\,\dd t +O(\varepsilon^2)\,.
\]
\end{corollary}

\begin{proof}
We calculate with \cref{lemma:surfacechange}
\begin{align*}
\varepsilon\int_0^1\int_{\partial S} \dd u_0(C_t \nu)\,\dd A\,\dd t &= \varepsilon\int_0^1\int_{\partial S} \dd u_0(\nu_t)\,g(\nu,C_t\nu)^{1/2}\,\dd A\,\dd t,\\
&= \varepsilon\int_0^1\int_{\partial S} \dd u_0(\nu_t)\,\dd A_t\,\dd t\,,\\
\intertext{and conclude with the divergence theorem and Fubini's theorem}
&= \varepsilon\int_0^1 \int_{S} \Delta_t u_0\,\mass\,\dd t = \varepsilon\int_{S}  \overline\Delta u_0\,\mass\,.
\end{align*}
The claim now follows from \cref{cor:eq6}.
\end{proof}

Using the transformation rules for normal vectors and surface forms from \cref{lemma:surfacechange} we can find the
representation of (the negative of) the leading-order total diffusive transport through a material boundary w.r.t.~an
arbitrary weighted manifold structure on the material manifold $(M,\tilde{g},\frac{\omega}{\dd\tilde{g}})$:
\begin{multline}\label{eq:arbitraryflux}
-\lim\limits_{\varepsilon \rightarrow 0}\tfrac{1}{\varepsilon}T_0^1(S, u_0) = -\overline T_0^1(S, u_0) = \int_0^1\int_{\partial S} \dd u_0(C_t \nu)\,\dd A\,\dd t = \\
\int_{\partial S} \dd u_0(\overline C \nu)\,\dd A = 
\int_{\partial S} \dd u_0(\overline C \tilde C^{-1}\tilde\nu)\,\dd \tilde A = \\
\int_{\partial S} \dd u_0(H\tilde\nu)\,\dd\tilde A = \int_{\partial S} \tilde g(\grad_{\tilde g}u_0,H\tilde\nu)\,\dd \tilde A\,,
\end{multline}
where $H=\overline g^{-1}\tilde{g}$ as claimed in \cref{eq:genericflux}.

\subsection{Relations to other dynamic surface areas}
\label{sec:surfaceconnection}

On a surface $\Gamma\subset M$ with $g$-unit normal vector field $\nu$, we compute
\begin{multline*}
\dd \overline A =  \sqrt{g\left(\nu, \overline C \nu\right)}\,\dd A =\\
= g\left(\nu, \left(\int_{\mathcal I} C_t \,\dd t\right)\nu\right)^{1/2}\,\dd A = \left({\int_{\mathcal I} g(\nu,C_t\nu)\,\dd t }\right)^{1/2}\,\dd A\,.
\end{multline*}
Plugging in \cref{timetarea} gives:
\[
\dd\overline A = \left({\int_{\mathcal I} \left(\frac{\dd A_{t}}{\dd A}\right)^2\,\dd t }\right)^{1/2}\,\dd A\,.
\]
This shows that the density of the surface element in the geometry of mixing w.r.t.~$\dd A$
is an $L^2$-average of the densities of the time-$t$ surface elements. Relating this with the interpretation in terms of diffusive transport, this is
consistent
with the observation made in \cite[Sect.~III.A]{Fyrillas2007}, ``that the rate of mass transport from an element of a material interface is related to the square of the relative change of the surface area''.

\begin{proposition}[Comparison to averages of surface areas]\label{trineq}
Let $\Gamma$ be a compact surface, and $\dd\overline A(\Gamma)$ and $\dd A_t(\Gamma)$ be its surface area as measured by $d\overline A$ and $\dd A_t$, respectively; i.e.,
\begin{align*}
\dd\overline A(\Gamma) &= \int_{\Gamma}\,\dd\overline A, & \dd A_t(\Gamma) &= \int_{\Gamma}\,\dd A_t.
\end{align*}
Then
\[
\dd\overline A(\Gamma) \geq \left(\int_{\mathcal I} \dd A_t(\Gamma)^2\,\dd t\right)^{\frac{1}{2}}\geq \int_{\mathcal I}\,\dd A_t(\Gamma)\,\dd t \eqqcolon \overline{\dd A_t (\Gamma)}\,.
\]
\end{proposition}

\begin{proof}
For convenience, we denote $\xi(t,p) = \frac{\dd A_t}{\dd A}(p)$ and compute
\begin{align}
\dd\overline A(\Gamma) &= \int_\Gamma \left(\int_{\mathcal I} \xi(t,p)^2\,\dd t\right)^{1/2} \,\dd A(p)
= \int_\Gamma \norm{\xi(\cdot,p)}_{L^2(\mathcal I)}\,\dd A(p)\nonumber\\
&\geq \norm{\int_\Gamma \xi(\cdot,p)\,\dd A(p)}_{L^2(\mathcal I)}\label{eq:prejensen}\\
&= \left(\int_{\mathcal I} \left(\int_\Gamma \xi(t,p)\,\dd A(p)\right)^2\,\dd t\right)^{1/2}
= \left(\int_{\mathcal I} \dd A_t(\Gamma)^2\,\dd t\right)^{1/2}\,,\nonumber
\end{align}
where \cref{eq:prejensen} is the triangle inequality for Banach-space valued maps (e.g. for the Bochner integral see \cite[Sect.~VI]{Lang1993}). The second claimed inequality is a direct consequence of Jensen's inequality applied to the expression in \cref{eq:prejensen}.
\end{proof}

Notably, $\overline{\dd A_t (\Gamma)}$ appears in the definition of the \emph{dynamic Cheeger constant} in \cite[Eq.~(4)]{Froyland2015a}.
Moreover, by means of the Cheeger inequality for weighted manifolds, \cref{prop:weightedcheeger} in \cref{sec:weightedcheeger}, we may strengthen the dynamic Cheeger inequality \cite[Thm.~3.2]{Froyland2015a}, where it was shown that
\[
\inf_{\Gamma}\frac{\overline{\dd A_t (\Gamma)})}{\min\{\mass(M_1),\mass(M_2)\}} \leq 2\sqrt{-\lambda_2}\,,
\]
for the case that $\omega = \dd g_t$ for all $t \in \mathcal I$. In this case, $\dd A_t$ is the $g_t$-Riemannian area.
Flat Riemannian manifolds were considered in \cite{Froyland2015a}, an extension to more general geometries was made in \cite{Froyland2017}.

\begin{corollary}[Strong dynamic Cheeger inequality]\label{cor:strong_Cheeger}
It holds
\[
\inf_{\Gamma}\frac{\overline{\dd A_t (\Gamma)})}{\min\{\mass(M_1),\mass(M_2)\}} \leq \inf_{\Gamma}\frac{\dd\overline A(\Gamma)}{\min\{\mass(M_1),\mass(M_2)\}} \leq 2\sqrt{-\lambda_2}\,,
\]
where $\inf_{\Gamma}$ denotes the infimum over all dividing surfaces $\Gamma$ that split $M$ into two sets $M_1$ and $M_2$, and $\lambda_2<0$ is the first non-trivial eigenvalue of $\dLap$.
\end{corollary}

\begin{proof}
The first estimate follows from \cref{trineq}, the second from \cref{prop:weightedcheeger}, since $\dd\overline A(\Gamma)/\min\set{\mass(M_1),\mass(M_2)}$ is the Cheeger constant for the geometry of mixing.
\end{proof}

\subsection{Relation to total Lagrangian diffusive transport}

The authors of \cite{Haller2018} establish the approximation of
the total diffusive flux as in \cref{cor:eq6}
in order to define a measure of diffusive permeability for a generic material surface $\Gamma$.
Here, the (``diffusive transport'') response $\overline T_0^1(\Gamma, u_0)$ to a
``diffusion stress'' given by some virtual initial condition $u_0$---of which $\Gamma$
is supposed to be a level set---is computed. As a consequence, the gradient of $u_0$
along $\Gamma$ is normal to $\Gamma$.

To make this construction comparable
among different surfaces, they require that the norm is uniformly constant
along the entire $\Gamma$, which specifies $u_0$ in a neighborhood of $\Gamma$
to first order. It remains to choose a norm w.r.t.~which to measure
the gradient and thereby require constancy. Since the response depends
linearly on this constant in the stress, one may take this constant to be equal to $1$
without loss of generality. The requirement on $u_0$ then reads as $\grad_gu_0=\nu$, with $\nu$ the $g$-unit normal along $\Gamma$. We set
\begin{equation}\label{eq:transport_g}
\overline T_0^1(\Gamma; g) \coloneqq  -\int_\Gamma \mathrm g(\nu ,\overline C \nu) \, \dd A \,,
\end{equation}
where, notationally, we replace the dependence of $\overline T_0^1$ on $u_0$ by a
dependence on the metric $g$ which determines (i) the gradient of $u_0$; (ii) the unit normal
vector; and (iii) the area element $\dd A$. By \cref{cor:Hallerflux}, the previous
definition equals the leading-order coefficient of $T_0^1(S, u_0)$ in the case that
$\partial S = \Gamma$ and $u_0$ is chosen as described above.

In \cite{Haller2018}, the reference metric $g$ is chosen as the one induced by the
initial spatial configuration of the fluid. For this choice, the norm of the gradient of
$u_0$ is constant as measured in the spatial metric.
This choice suggests itself, but is by no means natural. For instance, if at the initial
time instance the diffusion is not spatially homogeneous (along $\Gamma$), a
$u_0$ chosen with constant gradient measured w.r.t.~$g$ may have non-constant
gradient w.r.t.~$g_0$, the initial, diffusion-adapted metric. As a consequence, it will
have non-constant instantaneous diffusive flux, which puts different diffusion stress
on different subsets of $\Gamma$, and hence makes them incomparable.

Alternatively, one could argue that the gradient should be measured in the
``effective'' diffusion-adapted norm $\overline g$, the norm in the geometry of
mixing, and request uniform constancy w.r.t.~this norm; i.e.~ $\grad_{\overline g} u_0 = \overline\nu$ .
The diffusive transport represented in the geometry of mixing ($g=\overline g$), where $H$ is the identity (see \cref{eq:arbitraryflux}), reduces to
\begin{equation}\label{eq:negsamixing}
\overline{T}_0^1(\Gamma;\,\overline g) = -\int_{\Gamma}\overline g(\overline \nu,\overline \nu)\,\dd \overline A = - \int_{\Gamma} \dd \overline A\,,
\end{equation}
the (negative of the) surface area of $\Gamma$ in the geometry of mixing.
For comparison, we represent the surface area in the the geometry of the initial
configuration using \cref{lemma:surfacechange}, and obtain
\begin{equation}\label{eq:sa_mixing_onceagain}
\overline{T}_0^1(\Gamma;\,\overline g) = -\int_{\Gamma} \sqrt{g(\nu, \overline C \nu)}\,\dd A\,.
\end{equation}
and find that the different uniformization choices for (the gradient of) $u_0$ lead to
integrands that are the square and square root of each other, respectively.

Noticeably, within the $\overline{T}_0^1(\Gamma;\,\overline g)$ setting, the
problem of finding closed material surfaces that minimize leading-order diffusive
transport normalized by the enclosed fluid mass is exactly the isoperimetric problem
posed in the geometry of mixing; cf.~\cite{Froyland2015a,Froyland2017} for a
related but different approach (recall also \cref{cor:strong_Cheeger}, and the surrounding discussion).

\section{Conclusions}

In the above, we have investigated the $O(\varepsilon)$ asymptotics of
finite-time, time-dependent heat flow on manifolds as the diffusivity $\varepsilon$ goes to zero.
Such time-dependent heat flows arise naturally when studying (possibly time-dependent) advection-diffusion
equations in Lagrangian coordinates.
When the initial concentration $u_0$ is smooth
with support compactly contained in $M$, the behaviour of the advection-diffusion equation in leading order is
described by the time-averaged heat equation or, equivalently, the heat flow in the geometry of mixing.

The advection-diffusion equation remains well-defined even with non-smooth initial data $u_0$.
In particular, it seems natural to investigate $T_0^1(S, \mathds{1}_S)$, the diffusive transport out
of a material set $S$ when the initial density is uniformly distributed on $S$.
The theory developed in this work does not apply to this quantity. Here, the leading order asymptotics
is no longer of order $\varepsilon$, as even in the autonomous heat flow context
$T_0^1(S, \mathds{1}_S)$ is of order $\varepsilon^{1/2}$; see, e.g., \cite{Vandenberg2015,Schilling2020}.
There, the leading-order coefficient is proportional to the surface area of the boundary of $S$.
In the time-dependent, finite-time heat flow case, a similar result can be shown, where the relevant
surface area is the one in the geometry of mixing. This will be published in forthcoming work.

\subsection*{Acknowledgements}
This work is supported by the Priority Programme
SPP 1881 Turbulent Superstructures of the Deutsche Forschungsgemeinschaft.
We would like to thank Alvaro de Diego and Oliver Junge for fruitful discussions.

\appendix

\section{Differential geometric preliminaries}\label{sec:diffgeo}

In this section, we briefly recall some fundamental concepts from differential geometry and fix our notation. General references
include \cite{Lee2013,Grigoryan2009}. Throughout, let $M$ be a smooth,
oriented, compact manifold of dimension $\dim M=n$, possibly with smooth boundary.

A \emph{(Riemannian) metric $g$} on $M$ is a symmetric, positive-definite, contravariant
tensor field of rank 2, i.e., $g\colon TM\times TM\to\R$. For any tangent vector
$v \in T_xM$, a metric $g$ induces a linear form $g_x(v,\cdot)$ on $T_xM$.
Correspondingly, for any vector field $v$, the metric $g$ induces a one-form on $M$.
With the \emph{contraction operation}/\emph{interior multiplication} on forms, denoted
by $\iota$, i.e.,
\[
(\iota_F \alpha) (v_1,\dots,v_{k-1}) = \alpha(F,v_1,\dots,v_{k-1})\,,
\]
for any $k$-form $\alpha$, the induced one-form can be written as $w = \iota_v g$.
Henceforth, we identify a metric $g$ with its interpretation as the linear transformation
$g_x\colon T_xM \rightarrow T_x^*M$, $v \mapsto\iota_v g_x$, often referred to as the
\emph{canonical/musical isomorphism} between $T_xM$ and $T_x^*M$. Moreover, we will
often suppress the subscript $x$ and regard $g$ as a vector-bundle morphism
$g\colon TM \rightarrow T^*M$; cf., for instance, \cite{Lang1995}.

Non-degeneracy of $g$ implies its invertibility, and we may interpret its inverse $g^{-1}\colon T^*M\to TM$ by a similar identification as above
with a symmetric, positive-definite, covariant tensor field of rank 2, i.e., $g^{-1}\colon T^*M\times T^*M \to\R$.
This can be interpreted as an inner product on one-forms, and is known in the literature as the \emph{dual metric (to $g$)}.

With this notation, the gradient (induced by $g$) $\grad_g f$ is defined as the vector field  (a section of $TM$) obtained from
transforming the one-form $\dd f$ by $g^{-1}\colon T^*M\to TM$,
\begin{equation}\label{eq:gradient}
\grad_g f = g^{-1}\,\dd f\,.
\end{equation}

For any volume form $\omega$ on $M$, the induced \emph{divergence} $\divergence_\omega$
of a smooth vector field $F\colon M \rightarrow TM$ is defined via
\[
\left(\divergence_{\omega} F\right) \omega \coloneqq d(\iota_F \omega) = \mathcal L_V \omega,
\]
where $\divergence_{\omega} F \in C^\infty(M)$, and $\mathcal L$ is the Lie-derivative.

The induced \emph{(Riemannian) volume element} is the unique volume form, denoted by
$\dd g$ (the $\dd$ here does \emph{not} refer to the exterior derivative we used before),
that returns 1 when applied to an oriented, orthonormal set of tangent vectors
$v_1,\ldots,v_n\in T_xM$. It holds that $\divergence_{\dd g}$ is the usual Riemannian divergence.

Next, let $\Gamma$ be an oriented codimension-1 surface in $M$. Then the metric $g$ induces
a \emph{surface element} $\dd A_g$ on $\Gamma$ via the volume element on $M$ as follows.
For given, oriented linearly independent $v_1,\dots v_{n-1} \in T_p\Gamma$, let $\nu_g \perp_g \spn\{v_1, \dots, v_{n-1}\}$ with $\norm{\nu_g}_g = 1$ be such that $(\nu_g,v_1,\ldots,v_{n-1})$ is positively oriented in $M$. We call such $\nu$ \emph{the} unit normal vector to $\Gamma$ at $p$. Then the action of the surface element is given by
\begin{equation}\label{eq:areaform}
\dd A_g(v_1,\dots,v_{n-1}) = \dd g(\nu, v_1, \dots, v_{n-1}),\qquad v_1,\dots,v_{n-1}\in T_p\Gamma\,.
\end{equation}
Intuitively, $(v_1,\ldots,v_{n-1})$ span a parallelepiped of area 1 if, when expanded by the unit normal $\nu$,
the resulting parallelepiped has volume 1. By construction, a surface $\Gamma$ has non-negative \emph{surface area}
\[
\dd A_g(\Gamma) \coloneqq \int_{\Gamma}\,\dd A_g\,.
\]
The surface element $\dd A_g$ is a top-degree form on $\Gamma$ and can be hence regarded as the volume element there.

A natural differential operator on Riemannian manifolds is the \emph{Laplace-Beltrami operator} $\Delta_g$ defined as
\[
\Delta_{g} \coloneqq \divergence_{\dd g} \circ \grad_g.
\]

It will turn out that for an elegant description and study of a suitably general class of advection-diffusion processes,
\emph{weighted manifolds} (also known as \emph{manifolds with density} \cite{Morgan2005}) are very
helpful. A weighted manifold $(M, g, \theta)$ is a Riemannian manifold $(M, g)$, on which the volume
form and---as a consequence---the induced surface area forms are weighted by a (strictly)
positive smooth function $\theta\colon M \rightarrow \R$ w.r.t.~the canonical volume $\dd g$ or surface area
$\dd A_g$ forms \cite[Sect.~18.1]{Morgan2009b}.
For the induced surface area the same intuition and formalism applies: measure the volume of a higher-dimensional
parallelepiped as obtained by expansion with a suitably oriented unit normal vector, and the result is the area of the base parallelepiped.

On a weighted manifold $(M,g,\theta)$, the \emph{Laplace operator} $\Delta_{\theta,g}$ is defined analogously to the classic Riemannian case by
composition of the associated divergence and gradient,
\[
\Delta_{\theta,g} \coloneqq \divergence_{\theta\,\dd g}\circ\grad_g.
\]

\section{Cheeger inequality on weighted manifolds} \label{sec:weightedcheeger}

\begin{proposition}[Cheeger inequality for weighted manifolds]\label{prop:weightedcheeger}
Let $(M,g,\theta)$ be a compact weighted manifold with Laplace operator $\Delta$.
We denote the (weighted) volume form by $\mass \coloneqq \theta \dd g$, the (weighted) surface
measure by $\dd\overline A$, and the first nontrivial eigenvalue of $\Delta$ by $\lambda$.
Furthermore $\grad\coloneqq \grad_g$, and $\norm{\cdot} \coloneqq \norm{\cdot}_g$.
Then the Cheeger inequality holds:
\begin{equation}\label{cheegerinequality}
h\coloneqq \inf_{\Gamma\textnormal{ disconnects }M\textnormal{ into } M_1, M_2} \frac{\dd \overline A(\Gamma)}{\min\set{\mass(M_1), \mass(M_2)}} \leq 2\sqrt{-\lambda}\,.
\end{equation}
\end{proposition}

\begin{proof}
The proof of the classical Chee\-ger inequality given in \cite{Ledoux1994} applies---with obvious
modifications---to the weighted manifold case.
\end{proof}

\section{Spectral convergence}\label{sec:spectrum_proofs}

Recall that we used the notation $\Bra{\cdot,\cdot}_0$ for the $L^2(M,\mass)$ scalar product and
$\Bra{\cdot,\cdot}_1$ for the $H^1(M,g,\mass)$ scalar product; furthermore, we introduced $g$
as some reference metric on $M$ and $\grad = \grad_g$. For later reference, we first prove
estimates on solutions.

\begin{lemma}[Uniform parabolicity]\label{lemma:uniform_parabolicity}
There exist constants $C_1,C_2 > 0$ independent of $t$ so that
$C_1\abs*{u}_1^2 \leq -\Bra{u, \Delta_t u }_{0} \leq C_2 \abs*{u}_1^2$ for all $u \in H^1_0(M,g,\mass)$. Moreover, for $u_1, u_2 \in H^1_0(M,g,\mass)$ it holds that $\Bra{u_1,\Delta_t u_2} \leq C_2 \abs*{u_1}_1\abs*{u_2}_1$.
\end{lemma}

\begin{proof}
This is well-known to follow directly from uniform ellipticity of the smooth, $t$-de\-pen\-dent family of operators $\Delta_t$, defined on the compact $[0,1]$, which are in divergence form w.r.t.~the volume form $\mass$.
\end{proof}

For reference, we state the following well-known result/proof.

\begin{lemma}[$L^2$ contractivity; {\cite[Sect.~7.1, Thm.~2]{Evans2010}}]\label{lemma:l2contr}
Let $u_0 \in L^2(M,\mass)$. Then $\norm{P^\varepsilon_t u_0}_{0} \leq \norm{u_0}_{0}$.
\end{lemma}

\begin{proof}
To see this, note that
\[
\partial_t \norm{P^\varepsilon_t u}_{0}^2 = 2 \varepsilon \Bra{P^\varepsilon_t u, \Delta_t P^\varepsilon_t u}_{0} \leq 0\,,
\]
since $\Delta_t$ is non-positive. Absolute continuity of $t \mapsto \norm{P^\varepsilon_t u}_0^2$ is established in \cref{lemma:galerkin}, \cref{sec:parabolicpdes}.
\end{proof}

\begin{lemma}[{Uniform $H^1$ boundedness; cf.~\cite[Prop.~2(iii)]{Liu2004}}]\label{lemma:h11}
For $t \in [0,1]$ and $u_0 \in H^1_0$, we have $\abs*{P^\varepsilon_t u_0}_{1} \leq C_3 \abs*{u_0}_{1}$ for some constant $C_3$ that does not depend on $u_0$, $t$ or $\varepsilon$.
\end{lemma}

\begin{proof}
Our proof conceptually closely follows \cite[App.~B]{Liu2004}, which is given in Eulerian coordinates, and therefore takes a
seemingly different form because of the presence of the advection term in the evolution PDE.

We start with the case that $u_0$ is in the domain of $\Delta_0$. By uniform parabolicity it suffices to find bounds on $f^\varepsilon(t) \coloneqq -\Bra{u_\varepsilon(t), \Delta_t u_\varepsilon(t)}_{0}$.
Using \cref{lemma:regularity1}, see see that $f^\varepsilon(t)$ is absolutely continuous,
and moreover
\begin{align*}
\partial_t \Bra{ u_\varepsilon(t), \Delta_t u_\varepsilon(t)}_{0} &=2 \varepsilon  \Bra{\Delta_t u_\varepsilon(t), \Delta_t u_\varepsilon(t) }_{0} + \Bra{u_\varepsilon(t), \partial_t(\Delta_t) u_\varepsilon(t)}_{0}\\
&\geq \Bra{u_\varepsilon(t), \partial_t(\Delta_t) u_\varepsilon(t)}_0\,.
\end{align*}
The operator $\partial_t(\Delta_t)$ is given via its action on $u \in C^\infty(M)$ as
\begin{align*}
\partial_t(\Delta_t )u = \mathrm{div}_{\mass}(\partial_t (g_t^{-1}) \dd u )\,,
\end{align*}
recalling that $\Delta_t u = \mathrm{div}_{\mass}(g_t^{-1} \dd u)$ is the action of $\Delta_t$. Hence,
$\partial_t(\Delta_t )$ is a well-defined second-order partial differential operator with smooth coefficients.
Arguments as in the proof of \cref{lemma:uniform_parabolicity} yield that
\begin{align*}
C\coloneqq \sup\limits_{u \in H^1_0(M,\omega,g)}|\Bra{\partial_t(\Delta_t) u , u}_0|/|u|_{1}^2 < \infty\,.
\end{align*}
Therefore
\begin{align*}
-\partial_t \Bra{u_\varepsilon(t), \Delta_t u_\varepsilon(t)}_0 &\leq C|{u_\varepsilon(t)}|_1^2\,.
\end{align*}
Due to uniform parabolicity we have that $\abs{u_\varepsilon(t)}_1^2 \leq C_1^{-1}  f^\varepsilon(t)$.
Hence by Grönwall's lemma (\cite[Appendix B.2]{Evans2010}), $f^\varepsilon(t) \leq\mathrm{e}^{CC_1^{-1}} f^\varepsilon (0)$, which finishes the proof.
Since the domain of $\Delta_0$ is dense in $H^1_0$, the general result is a consequence of this special case (using the results in \cref{sec:parabolicpdes}).
\end{proof}

Recall the well-known fact that the $L^2$-adjoint of $P^\varepsilon$ is the time-$1$ solution
operator associated to the Lagrangian advection-diffusion equation with the same Dirichlet boundary
conditions, but with reversed time\footnote{See, for example, the proof of
\cite[Prop.~2.9]{Acquistapace1991} and \cref{sec:parabolicpdes}}, i.e.,
\begin{equation}\label{eq:adjointproblem}
\partial_t u(t,x) = \varepsilon\Delta_{(1-t)} u(t,x)\,.
\end{equation}
The range of $(P^\varepsilon)^*$ is a subset of $H^1_0(M,g,\mass)$; see \cref{sec:parabolicpdes}.
Therefore, the left singular vectors of $P^\varepsilon$, or equivalently the eigenvectors of $(P^\varepsilon)^*P^\varepsilon$, are in $H^1_0(M,g,\mass)$.
Recall that the constant $C_3$ from \cref{lemma:h11} depends (i) on the uniform parabolicity bounds $C_1$ and $C_2$ from \cref{lemma:uniform_parabolicity},
and (ii) on bounds on $\partial_t(\Delta_t)$. All of these bounds equally apply to \cref{eq:adjointproblem}. Therefore, we conclude with \cref{lemma:h11} that
\begin{equation}\label{eq:adjointbound}
\abs{P^\ast u}_1 \leq C_3 \abs{u}_1\,,
\end{equation}
for $u \in H^1_0(M,g,\omega)$. Furthermore, the same estimate
\begin{align}
\abs{P_{t,1} u}_1 &\leq C_3 \abs{u}_1\label{eq:timetbound}\,,
\end{align}
applies to the solution operator (from time $t$ to time 1) of the Lagrangian advection-diffusion equation,
considered on the time interval $[t,1]$.
By construction,
\[
P_1^\varepsilon = P_{t,1}^\varepsilon P_{t}^\varepsilon
\]
for any $t \in (0,1)$.

\begin{lemma} \label{lemma:reverseineq}
There exists $C_4 > 0$, independent from $\varepsilon$, satisfying
\[
|v^\varepsilon|_1 \leq \max\limits_{t \in [0,1]} \abs*{v^\varepsilon(t)}_1 \leq C_4 \min\limits_{t \in [0,1]} \abs*{v^\varepsilon(t)}_{1} \leq C_4\abs*{v^\varepsilon}_{1}
\]
for sufficiently small $\varepsilon$ and the singular vector $v^{\varepsilon}$.
Recall that $v^\varepsilon(t) \coloneqq P^\varepsilon_t v^\varepsilon$.
\end{lemma}
\begin{proof}
The rightmost and leftmost inequalities are trivial. For the middle inequality, by \cref{lemma:h11} we have
\begin{align} \label{eq:followsfromh11}
\max\limits_{t \in [0,1]} \abs*{v^\varepsilon(t)}_1 \leq C_3 \abs*{v^\varepsilon}_1\,.
\end{align}
Thus it is enough to show
\begin{align}\label{eq:uniformboundssufficient}
\abs*{v^\varepsilon}_1 \leq C \abs*{v^\varepsilon(t)}_1
\end{align}
for all $t \in [0,1]$, with some $C > 0$ independent from $t$ or $\varepsilon$.
Since the square of singular values of $P^\varepsilon$ are eigenvalues of $(P^\varepsilon)^\ast P^\varepsilon$, we have
\begin{equation}\label{eq:reverseineqalmost}
(\sigma^\varepsilon)^2 v^\varepsilon(0) = (P^\varepsilon)^\ast P^\varepsilon_{1} v^\varepsilon(0) = (P^\varepsilon)^\ast v^\varepsilon(1)\,.
\end{equation}
Applying \cref{eq:timetbound} to $v^\varepsilon(1) = P^\varepsilon_{t,1} v^\varepsilon(t)$ yields $\abs*{v^\varepsilon(1)}_1 \leq C_3 \abs*{v^\varepsilon(t)}_1$.
\Cref{eq:adjointbound,eq:reverseineqalmost} yield that $(\sigma^\varepsilon)^2 \abs*{v^\varepsilon}_1 \leq C_3\abs*{v^\varepsilon(1)}_1$.
Combining these inequalities, we obtain $\abs*{v^\varepsilon(0)}_{1} \leq (\sigma^\varepsilon)^{-2} C_3^2 \abs*{v^\varepsilon(t)}_1$.
We know that (step 1 of \cref{thm:spectralconvergence}) $\sigma^\varepsilon \rightarrow 1$ for $\varepsilon\to 0$, and is thus bounded away from zero for sufficiently small $\varepsilon$. This proves \cref{eq:uniformboundssufficient}, and the claim is shown.
\end{proof}

\begin{lemma}[$H^1$ bound on singular vectors] \label{lemma:h1bound}
There exists a constant $C > 0$, independent of $\varepsilon$ and $t$,
for which $ |{v^\varepsilon(t)}|_1 \leq C$ holds for $t \in [0,1]$ and sufficiently small $\varepsilon$.
\end{lemma}

\begin{proof}
With \cref{eq:importanteq,lemma:uniform_parabolicity,lemma:reverseineq} we obtain for any $t \in (0,1]$ and sufficiently small $\varepsilon$ that:
\begin{align*}
\frac{(1 - (\sigma^\varepsilon)^2)}{2 \varepsilon} &= -\int_0^1 \Bra{v^\varepsilon(t), \Delta_t v^\varepsilon(t)}_{0}\,\dd t  \\
&\geq \min\limits_{t \in [0,1]} - \Bra{v^\varepsilon(t), \Delta_{t} v^\varepsilon(t)}_{0}\\
&\geq C_1 \min\limits_{t \in [0,1]} |v^\varepsilon(t)|^2_1 \\
&\geq C_1 C_4^{-1} \abs{v^\varepsilon(0)}_1^2\,.
\end{align*}
In \cref{eq:squareeq} we have already shown that the  limit superior of the left hand side is less than or equal to $-\overline{\lambda}$, and,
therefore, it may be bounded from above by, say, $-2\overline \lambda$ for sufficiently small $\varepsilon$.
This shows that $\abs{v^\varepsilon}_1^2 \leq -2 \overline \lambda C_4 C_1^{-1}$, proving the claim for $t=0$. The case $t \neq 0$ is now a consequence of \cref{lemma:h11}.
\end{proof}

\section{Parabolic PDEs on compact manifolds with boundary} \label{sec:parabolicpdes}

We briefly collect some properties of second-order parabolic PDEs on compact and orientable smooth
Riemannian manifolds with (potentially empty) $C^2$ boundary. These properties are well known when the domain
is an open subset of $\R^n$ \cite{Evans2010,Renardy2004} and the straightforward extension to
compact manifolds seems to be folklore knowledge , though rarely explicitly treated; see \cite{Amann2016}.
We describe below some properties of the the Galerkin-approach described in \cite{Evans2010}\cite[Chapter 11.1]{Renardy2004} with straightforward modification to the time-dependent mass-preserving setting 
on a compact manifold; the reasoning below is included only in order to (a) demonstrate that well-known results on $\mathbb R^n$ indeed extend to compact manifolds in a straightforward way because we could not find a reference and (b) collect some technical results arising directly in the standard Galerkin approach that we require elsewhere.

Let $\mass$ be a smooth, nonvanishing volume-form on $M$. For convenience, we will use a metric
$g$ such that $\dd g = \mass$. The metric may be constructed by any metric on $M$ after suitable rescaling.
We need this metric only for defining a norm on $H^1(M,g)$, given by
\begin{align*}
\norm{u}_{H^1(M,g)}^2 \coloneqq \int_M g(\grad u, \grad u) \mass + \int_M |u|^2 \mass\,,
\end{align*}
where $\grad u$ is interpreted in a suitably weak sense. Since $M$ is compact, the specific choice of $g$ will not affect the topology of $H^1(M,g)$.
The space $H_0^1(M,g)$ is defined as the completion of $\cicm$ w.r.t.~$\norm{\cdot}_{H^1(M,g)}$ \cite{Grigoryan2009,Jost2011}.

We will describe the parabolic PDE theory needed for the equation
\begin{align}\label{parabolicprob}
\partial_t u &= \divergence_{\mass}( D(t) \dd  u)\,,
\end{align}
with $D\colon [0,1] \times T^*M \rightarrow TM$ a smooth---including at the boundary---family of nonvanishing
bundle morphisms that are symmetric in the sense that $D(t,u)(v) = D(t,v)(u)$ for $t \in [0,1]$ and all vector fields $u$ and $v$.
Let $L(t)v \coloneqq \divergence_{\mass}(D(t)\dd u)$. The tensor field $D$ is bounded---due to its smoothness and compactness of $M$---and
nonvanishing, hence the operator $\partial_t - L$ is \emph{uniformly parabolic}, i.e., there exists $\alpha > 0$ such that for any $v \in H^1_0(M,g)$
\begin{align}
\alpha^{-1}\norm{\grad_g v}^2_{L^2(M,\mass)} \leq -\Bra{v,L(t) v}_{L^2(M,\mass)} \leq \alpha \norm{\grad_g v}^2_{L^2(M,\mass)}\,.
\end{align}

For what follows, we require the well-known theory of vector-valued So\-bo\-lev spaces, and our notation essentially follows \cite{Renardy2004};
see also \cite[Appendix A]{Brezis1973} for proofs of fundamental results.
For a Hilbert space $X$, we write $X^*$ for its dual, and $H^{-1}(M,g) \coloneqq H^1_0(M,g)^*$.

As in \cite{Renardy2004}, to each $t \in [0,1]$ we associate an operator $L(t)\colon H^1_0(\mass) \mapsto H^{-1}(M,g)$,
defined by
\[
(L(t)u)v = \Bra{\dd u, D(t) \dd v}_{L^2(M,\mass)} = \int_M \dd u(D(t) \dd v) \,\mass\,.
\]
The space $L^2(M,\mass)$ embeds continuously into $H^{-1}$ by the identification of a function $f \in L^2(M,\omega)$ with the functional $\Bra{f,\cdot}_{L^2(M,\mass)}$.
By a slight abuse of notation, for $f \in H^{-1}(M,g)$ and $g \in H^1_0(M,g)$ we will write $\Bra{f,g}_{L^2(M,\mass)} \coloneqq f(g)$,
even if $f$ is not contained in the image of the embedding.

\begin{lemma}\label{lemma:galerkin}
\Cref{parabolicprob} has a unique weak solution
\[
u \in C([0,1];L^2(M,\mass)) \cap L^2((0,1); H^1_0(M,g)),
\]
given an initial value $u(0,\cdot) \in L^2(M,\mass)$. Moreover, the function $t \mapsto \norm{u(t,\cdot)}^2_{L^2(U)}$ is absolutely continuous, with
\begin{align}
\frac{d}{dt}\norm{u(t)}^2_{L^2(M,\mass)} &= 2\Bra{L(t)u(t), u(t) }_{L^2(M,\mass)}
\end{align}
for almost all $t$, where the right hand side must be interpreted in a weak sense.
\end{lemma}

\begin{proof}
The $L^2$-Galerkin approach described in \cite[Sect.~7.1, Thms.~3 \& 4]{Evans2010} yields existence and uniqueness for the compact manifold case 
just like for $M \subset \R^n$ compact. 
Theorem 3 in \cite[Sect.~5.9]{Evans2010} proves the remaining claims.
\end{proof}

These arguments show that for $t \in [0,1]$, the time-$t$ solution operator to \cref{parabolicprob} is well-defined when viewed as an operator $P_t \colon L^2 \rightarrow L^2$.
Arguments as in \cref{lemma:l2contr} establish its continuity.

Let the domain of $L(t)$ be the collection of $f \in H^1_0(M,g)\subset L^2(M,\mass)$ satisfying $L(t)f \in L^2(M,\mass)$. As a consequence of elliptic regularity theory and the fact that we are working with homogeneous Dirichlet boundary (see \cite[Chapters 6.3 \& 7.4]{Evans2010}), 
this function space does not depend on $t$.
By arguments as in \cite[Chapter 7]{Pazy1983}, one sees that for $t \in (0,1]$, the image of the time-$t$ solution operator $P_t$ is in the domain of $L(t)$.
Hence, the image of $P_t$ is in $H^1_0(M,g)$ for all $t \in (0,1]$. Thus, the operator $P_t\colon L^2(M,\mass) \to H^1_0(M,g)$ is well-defined, and by the
closed graph theorem it is continuous. By the Rellich-Kondrachev theorem, $P_t$ is therefore compact when viewed as an operator from $L^2(M,\mass)$ to itself.

\begin{lemma}\label{lemma:regularity1}
Provided that the initial value $u_0$ is in the domain of $L(0)$, the solution from \cref{lemma:galerkin} is sufficiently regular such that
\begin{enumerate}[(i)]
	\item $u \in H^1\left((0,1); H^1_0(M,g)\right)$,
	\item $L(t)u \in H^1\left((0,1);H^{-1}(M,g)\right)$, and
	\item $L(t)u \in C\left([0,1];L^2(M,\mass)\right)$.
\end{enumerate}
The function $\Bra{u,L(t)u}_{L^2(M,\mass)}$ is absolutely continuous with 
\begin{multline}\label{eq:weakdiff2}
\frac{d}{dt}\Bra{u(t), L(t) u(t) }_{L^2(M,\mass)} = \\2 \Bra{L(t)u(t),L(t)u(t)}_{L^2(M,\mass)} + \Bra{L'(t)u(t),u(t)}_{L^2(M,\mass)}
\end{multline}
for almost all $t \in [0,1]$.
\end{lemma}
\begin{proof}
Proceed as in \cite[Sect.~11.1.4]{Renardy2004}, for the last statement 
a result like \cite[Cor.~A.4]{Brezis1973} is required.
\end{proof}

\begin{lemma}
If $u_0 \in \cicm$, then the solution $u$ from \cref{lemma:galerkin} is in $C^\infty([0,1] \times M)$.
\end{lemma}
\begin{proof}
Certainly $u_0$ is in the domain of all powers of $L(t)$. Iterating the construction of \cite[Sect.~11.1.4]{Renardy2004} 
together with Sobolev embedding and elliptic-regularity results yields the claim. See also \cite[Sect.~7.1, Thm.~7]{Evans2010}
for a proof of the nonautonomous case on open subsets of $\R^n$.
\end{proof}

We conclude with a well-known property of smooth solutions to parabolic equations.

\begin{lemma}[{Weak maximum principle on manifolds; \cite[Thm.~A.3.1]{Jost2011} or \cite[Sect.~7.1, Thm.~8]{Evans2010}}] Let $u \in C^{1,2}([0,1] \times M) \cap C([0,1] \times \overline M)$. If $\mathcal L^\varepsilon u \leq 0$ on $[0,1] \times \mathrm{int}(M)$, then for the ``parabolic boundary'' $B \coloneqq [0,1] \times \partial M\cup \{0\} \times M$ one has
\label{lemma:wmp}
\begin{align}
\max_{[0,1] \times M} u &= \max_B u\,.
\end{align}
\end{lemma}


\end{document}